\documentclass[10pt]{amsart}
\usepackage{amscd}
\usepackage[arrow,matrix]{xy}
\usepackage{graphicx}
\usepackage{amsmath}
\usepackage{amsmath, latexsym, amssymb}
\input xypic
\numberwithin{equation}{section}
\theoremstyle{plain}
\newtheorem{lemma}{Lemma}[section]
\newtheorem{proposition}[lemma]{Proposition}
\newtheorem{theorem}[lemma]{Theorem}
\newtheorem{corollary}[lemma]{Corollary}

\theoremstyle{definition}
\newtheorem{definition}[lemma]{Definition}
\newtheorem{remark}[lemma]{Remark}
\newtheorem{example}[lemma]{Example}

\begin{document}
\newcommand{\R}{{\mathbb R}}
\newcommand{\C}{{\mathbb C}}
\newcommand{\F}{{\mathbb F}}
\renewcommand{\O}{{\mathbb O}}
\newcommand{\Z}{{\mathbb Z}} 
\newcommand{\N}{{\mathbb N}}
\newcommand{\Q}{{\mathbb Q}}
\renewcommand{\H}{{\mathbb H}}

\newcommand{\Aa}{{\mathcal A}}
\newcommand{\Bb}{{\mathcal B}}
\newcommand{\Cc}{{\mathcal C}}    
\newcommand{\Dd}{{\mathcal D}}
\newcommand{\Ee}{{\mathcal E}}
\newcommand{\Ff}{{\mathcal F}}
\newcommand{\Gg}{{\mathcal G}}    
\newcommand{\Hh}{{\mathcal H}}
\newcommand{\Kk}{{\mathcal K}}
\newcommand{\Ii}{{\mathcal I}}
\newcommand{\Jj}{{\mathcal J}}
\newcommand{\Ll}{{\mathcal L}}    
\newcommand{\Mm}{{\mathcal M}}    
\newcommand{\Nn}{{\mathcal N}}
\newcommand{\Oo}{{\mathcal O}}
\newcommand{\Pp}{{\mathcal P}}
\newcommand{\Qq}{{\mathcal Q}}
\newcommand{\Rr}{{\mathcal R}}
\newcommand{\Ss}{{\mathcal S}}
\newcommand{\Tt}{{\mathcal T}}
\newcommand{\Uu}{{\mathcal U}}
\newcommand{\Vv}{{\mathcal V}}
\newcommand{\Ww}{{\mathcal W}}
\newcommand{\Xx}{{\mathcal X}}
\newcommand{\Yy}{{\mathcal Y}}
\newcommand{\Zz}{{\mathcal Z}}

\newcommand{\zt}{{\tilde z}}
\newcommand{\xt}{{\tilde x}}
\newcommand{\Ht}{\widetilde{H}}
\newcommand{\ut}{{\tilde u}}
\newcommand{\Mt}{{\widetilde M}}
\newcommand{\Llt}{{\widetilde{\mathcal L}}}
\newcommand{\yt}{{\tilde y}}
\newcommand{\vt}{{\tilde v}}
\newcommand{\Ppt}{{\widetilde{\mathcal P}}}
\newcommand{\bp }{{\bar \partial}} 
\newcommand{\ad}{{\rm ad}}
\newcommand{\Om}{{\Omega}}
\newcommand{\om}{{\omega}}
\newcommand{\eps}{{\varepsilon}}
\newcommand{\Di}{{\rm Diff}}

\renewcommand{\a}{{\mathfrak a}}
\renewcommand{\b}{{\mathfrak b}}
\newcommand{\e}{{\mathfrak e}}
\renewcommand{\k}{{\mathfrak k}}
\newcommand{\pg}{{\mathfrak p}}
\newcommand{\g}{{\mathfrak g}}
\newcommand{\gl}{{\mathfrak gl}}
\newcommand{\h}{{\mathfrak h}}
\renewcommand{\l}{{\mathfrak l}}
\newcommand{\sm}{{\mathfrak m}}
\newcommand{\n}{{\mathfrak n}}
\newcommand{\s}{{\mathfrak s}}
\renewcommand{\o}{{\mathfrak o}}
\newcommand{\so}{{\mathfrak so}}
\renewcommand{\u}{{\mathfrak u}}
\newcommand{\su}{{\mathfrak su}}

\newcommand{\ssl}{{\mathfrak sl}}
\newcommand{\ssp}{{\mathfrak sp}}
\renewcommand{\t}{{\mathfrak t }}
\newcommand{\Cinf}{C^{\infty}}
\newcommand{\la}{\langle}
\newcommand{\ra}{\rangle}
\newcommand{\half}{\scriptstyle\frac{1}{2}}
\newcommand{\p}{{\partial}}
\newcommand{\notsub}{\not\subset}
\newcommand{\iI}{{I}}               
\newcommand{\bI}{{\partial I}}      
\newcommand{\LRA}{\Longrightarrow}
\newcommand{\LLA}{\Longleftarrow}
\newcommand{\lra}{\longrightarrow}
\newcommand{\LLR}{\Longleftrightarrow}
\newcommand{\lla}{\longleftarrow}
\newcommand{\INTO}{\hookrightarrow}

\newcommand{\QED}{\hfill$\Box$\medskip}
\newcommand{\UuU}{\Upsilon _{\delta}(H_0) \times \Uu _{\delta} (J_0)}
\newcommand{\bm}{\boldmath}
\title[Orbits  in  real $\Z_m$-graded  semisimple     Lie algebras]{\large
Orbits in  real $\Z_m$-graded semisimple Lie algebras}
\author{H\^ong V\^an L\^e}
\date{\today}

\medskip

\abstract  
In this note we  propose a method to classify    homogeneous nilpotent  elements
in a real $\Z_m$-graded semisimple Lie algebra $\g$. Using this we describe  the set of orbits 
of homogeneous elements in a  real $\Z_2$-graded semisimple Lie algebra. A classification of     4-vectors (resp.  4-forms) on $\R^8$ can be given  using this method.
\endabstract

\maketitle
\tableofcontents



\medskip

\tableofcontents

\section {Introduction}

Let $\g = \oplus_{i\in\Z_m} \g_i$ be a real $\Z_m$-graded   semisimple Lie algebra. If $m \ge 3$ we cannot associate
to this $\Z_m$-gradation  a  compatible finite order automorphism  of $\g$ as in the case of complex  $\Z_m$-graded Lie algebras,
unless $m$ is even and the only nonzero components  of $\g$ have  degree $0$ or $m/2$. To get around this  problem we  extend the $\Z_m$-gradation  on $\g$ linearly to a  $\Z_m$-gradation on  the complexification $\g ^\C$.
Denote by  $\theta^\C $  the  automorphism of $\g ^\C$ associated  with this $\Z_m$-gradation, i.e. $\theta^\C_{ |\g _k^\C} = \exp {{2\pi \sqrt{-1}k\over m}}\cdot Id $. 

Let $G^\C$ be the connected simply-connected Lie group whose Lie algebra is $\g^\C$. Clearly, $\theta^\C$  can be lifted to an  
automorphism $\Theta^\C$ of $G^\C$. Denote by $G _0^\C$   the connected  Lie subgroup in $G^\C$ whose Lie algebra
is $\g_0^\C$.  A result by Steinberg \cite[Theorem 8.1]{Steinberg1968}  implies that $G_0^\C$ is the  Lie subgroup  consisting of fixed  points of $\Theta^\C$.  Note that the  adjoint action of group $G_0^\C$   on $\g^\C$ preserves the induced  $\Z_m$-gradation on $\g ^ \C$.  Let $G$ be the  connected Lie subgroup in $G ^\C$ whose Lie algebra is $\g$. Denote
by $G_0$ the connected  Lie subgroup in $G$ whose Lie algebra is $\g_0$. The adjoint action of $G_0$ on
$\g$ preserves the $\Z_m$-gradation. We note that the adjoint action of $G_0$ on $\g$   coincides with  the adjoint action
of any connected Lie subgroup $\tilde G_0$  of a connected Lie group  $\tilde G$ having  
Lie algebras $\g_0$   and $\g$ correspondingly.  In \cite{Vinberg1976} Vinberg observed that      
by considering a new $\Z_{\bar m}$-graded Lie algebra $\bar \g$, $\bar m = {m\over (m,k)}$ and
$\bar \g_p = \g_{pk}$ for  $p \in \Z_{\bar m}$ we can regard the  adjoint action of $G_0$ on $\g_k$ as
the action of $G_0$ on $\bar \g_1$.   Thus in this note we will consider only the adjoint action of
$G_0$ on $\g_1$. We also write ``the adjoint action/orbit(s)", or simply ``orbits", if no misunderstanding  can occur.

The problem of  classification of  the adjoint orbits in real or complex  graded semisimple Lie algebras  $\g = \oplus_{i\in \Z_m} \g_i$  is related to many important algebraic and geometric questions. In  \cite{Vinberg1975}
Vinberg proposed a   method  to classify  the adjoint orbits  in complex  $\Z_m$-graded semisimple Lie algebras.  His  work  developed further   the theory of  $\Z_2$-graded  complex semisimple  Lie algebras   by Kostant and Rallis \cite{KR1971},  and  the theory of finite order automorphisms  on complex simple Lie algebras by Kac \cite{Kac1969}.
It is known that all Cartan subspaces in $\g_1^\C$ are conjugate \cite{Vinberg1976}. 
Thus the classification
of  semisimple elements in $\g_1 ^\C$ is reduced to the classification of the orbits of the associated Weyl group on a Cartan subspace  in $\g_1 ^ \C$ \cite{Vinberg1976}.  To classify  nilpotent elements  in  $\g_1 ^\C$, Vinberg  proposed a  method  of support, which associates to each nilpotent element $e$ in $\g_1$ a  $\Z$-graded  semisimple Lie  algebra
defined by  a characteristic $h(e)$ of $e$, see section 4 for more details. In a complex $\Z_m$-graded  semisimple Lie algebra a nilpotent element $e$ in $\g_1$ is defined uniquely up to conjugacy with respect to the centralizer of    $h(e)$ \cite{Vinberg1975}. 
If $m =1$,  we can also classify  nilpotent orbits   in a simple Lie algebra $\g$  over an algebraic closed field of characteristic 0, or of prime characteristic $p$, provided  $p$ is sufficient large. We refer the reader to the book by Collingwood and McGovern \cite{C-M1993} and the book by Humphreys \cite{Humphreys1995}  for  surveys.

 In a real $\Z_m$-graded  semisimple Lie algebras $\g$  the  conjugacy classes of  Cartan subspaces  may  consist of more than one element. Furthermore, a given characteristic  element   in  a real  $\Z_m$-graded Lie algebra can be  associated with many
 conjugacy classes  of nilpotent elements in $\g_1$. These phenomena   are main difficulties when we    want to classify
 the adjoint orbits in a real $\Z_m$-graded semisimple Lie algebra $\g$.
If $m =1$, i.e. $\g$ is without gradation,  a classification
of the adjoint orbits  of nilpotent elements in $\g$ can be  obtained, using the Cayley transform
\cite{Djokovic1987}, \cite{Sekiguchi1987} and  a  classification of nilpotent
elements in   the associated $\Z_2$-graded complex  semisimple Lie algebra, see e.g. \cite{C-M1993}, \cite{Djokovic1988}.
Furthermore,  a classification of the adjoint orbits of semisimple elements in  $\g$ can be   obtained  from the classification  of  Cartan  subalgebras in $\g$ by
Kostant \cite{Kostant1955} and Sugiura \cite{Sugiura1959}. We also like to mention here the  work by Rothschild
on the   adjoint orbit space  in a real reductive algebra \cite{Rothschild1972}, as well as  the work by Djokovic
on the  adjoint orbits  of  nilpotent elements in  $\Z$-graded Lie algebra $\e_{8(8)}$ \cite{Djokovic1983}. An essential part of our method
of classification of nilpotent orbits in real $\Z_m$-graded  semisimple Lie algebras   is a combination of certain ideas in  their works. 

In this note we propose a method to classify the adjoint orbits of homogeneous  nilpotent elements
in a real $\Z_m$-graded semisimple Lie algebra $\g$.  Roughly speaking, our method of classification of homogeneous nilpotent elements in $\g$ consists of two steps. In the first step we classify  the conjugacy classes of  characteristics in a given
real $\Z_m$-graded  semisimple Lie algebra. In the second step we  classify the conjugacy classes of nilpotent elements  associated with a
given conjugacy class of a characteristic.  The first step uses the Vinberg classification of characteristics in the complexification $\g_1 ^ \C$ \cite{Vinberg1979} combining with the Djokovic  classification of real  forms of  a given complex $\Z$-graded semisimple Lie algebra \cite{Djokovic1982}, taking into account
our observation that there is an  injective map from the set of  $Ad_{G_0}$-conjugacy classes of
 characteristics  in $\g_0$ to the set of    $Ad_{G_0 ^ \C}$-conjugacy classes  of characteristics in $\g_0 ^ \C$, see Lemma \ref{Lemma.5.1} and  Remark \ref{Remark.5.2}.  To perform the second step we analyze the set of singular elements in a real $\Z$-graded  semisimple Lie algebra defined by a given characteristic, see section 4 for more details.  It turns out that we can apply algorithms in real algebraic geometry to distinguish  the  conjugacy classes of nilpotent elements
 associated with a given  characteristic.   Our recipe to  classify  nilpotent elements is
 summarized in Remark \ref{alg2}. We note that the related algorithm  in  real algebraic geometry is highly complicated. To  apply  our algorithm for interesting cases  we will need a powerful computer system   together with a  suitable software, see Remark \ref{rsa1}.
 
For $m = 2$  a classification of Cartan subspaces in $\g_1$ has been obtained by Oshima and
Matsuki \cite{OM1980}.  Using their
classification and our results in previous section, we describe the set of orbits of    homogeneous elements  of degree 1 in  a $\Z_2$-graded  semisimple Lie algebra, following
 the  same scheme proposed by Elashvili and Vinberg in \cite{EV1978}, see Remark \ref{sumix}.

The plan of our note is as follows. In section  2 we recall main notions and  prove   a version of the Jacobson-Morozov-Vinberg theorem for real $\Z_m$-graded semisimple Lie algebras, see Theorem \ref{Theorem.2.1}.
In section 3 we prove the existence of a R-compatible Cartan involution  on $\g = \oplus _{i \in \Z_m} \g_i$,
which   provides us   an isomorphism between the $Ad_{G_0}$-orbit spaces on $\g_i$ and $\g_{-i}$, see Corollary \ref{rev}. We also
give many important  examples of   real $\Z_m$-graded semisimple Lie algebras in this section. In section 4 we  propose a method
to classify homogeneous nilpotent elements in  a real  $\Z_m$-graded semisimple Lie algebra. We demonstrate our method in Example \ref{lin}. In section 5 we  describe the set of 
homogeneous elements  in a   real $\Z_2$-graded semisimple Lie algebra. In this section we also explain the relation between    a classification of homogeneous elements in real $\Z_m$-graded  semisimple Lie algebras and
 a classification of   k-vectors (resp. k-forms) on $\R ^8$.

\medskip

\section {Semisimple elements and  nilpotent elements of a  real $\Z_m$-graded  semisimple 
Lie algebra}

\medskip

Let $\g =  \oplus_{i\in \Z_m} \g_i$ be  a real $\Z_m$-graded  semisimple  Lie algebra. 
An element $x\in \g _i$, $i = \overline{0,m-1}$,  is  called {\it semisimple} (resp. {\it nilpotent}), if $x$ is  semisimple  (resp. nilpotent)  in $\g$. 
In this section we explain the   Jordan decomposition   for an element $x \in \g_i$. We  also prove an analog of the Jacobson-Morozov-Vinberg theorem on the existence of  an $\ssl_2$-triple associated to a homogeneous nilpotent element in $\g_1$, see Theorem \ref{Theorem.2.1}, and we introduce the notion of a Cartan subspace in $\g_1$. 
 
 \medskip
 
{\bf  Jordan  decomposition in  a real $\Z_m$-graded semisimple   Lie algebra.} {\it Any  $x\in \g_i$  has
a unique decomposition $x= x_s + x_n$, where  $x_s$ is semisimple, $x_n$ is nilpotent, $x_s, x_n \in \g _i$,    $[x_s, x_n] = 0$.}

\medskip


   For a real form $\g$ of $ \g ^\C$ let us denote by $\tau _\g$  the complex conjugation of $\g ^\C$ with respect to $\g$. It is  easy to see that the existence
and the uniqueness of the  Jordan  decomposition   for  $x\in \g_i$ follows from the existence and the uniqueness of
the Jordan decomposition  for $x $ in $ \g_i ^\C$ \cite{Vinberg1976}, since this decomposition  is invariant under the complex conjugation  $\tau _\g$,  which  preserves the  $\Z_m$-gradation on $\g ^\C$.

The case  $m=1$  has been treated before, see  e.g. \cite[chapter IX, exercise A.6]{Helgason1978},  and  the references
therein. 

\medskip 

The  following Theorem \ref{Theorem.2.1} is an analogue of the  Jacobson-Morozov-Vinberg theorem in \cite[Theorem 1(1)]{Vinberg1979}.  Some partial cases of  Theorem 
\ref{Theorem.2.1} has been proved in \cite[Lemma 6.1]{Djokovic1983},  and in \cite[Theorem 9.2.3]{C-M1993}.

For any element $e\in \g$ let us denote 
 by $\Zz _{ G_0}(e)$ the centralizer of $e$ in $G_0$.

\begin{theorem}[Jacobson-Morozov-Vinberg (JMV)  theorem  for a real $\Z_m$-graded semisimple Lie algebra]\label{Theorem.2.1}   Let $e \in \g _1$ be  a  nonzero  nilpotent element.\\ 
i) There is a semisimple element $h \in \g _0$ and a nilpotent element  $f \in \g _{-1}$   such that
$$ [h, e] = 2e, \, [h, f] = -2f, \, [e,f] = h.$$
ii) Element $h$ is defined  uniquely up to  conjugacy via an element in
$\Zz _{ G _0} (e)$.\\
iii)  Given $e$ and $h$, element $f$ is defined uniquely.
\end{theorem}

\medskip

\begin{remark}\label{Remark.2.2} -The JMV Theorem plays a key role  in the study of nilpotent elements. This Theorem associates to each nilpotent element $e$  a semisimple element $h \in \g_0$, which is defined by $e$ uniquely up  to  conjugation. The element $h$ in Theorem \ref{Theorem.2.1}  is called {\it characteristic} (or  a characteristic) of $e$. We also denote  a characteristic  of $e$  by $h(e)$. We call an element $h\in \g_0$   characteristic, if it is a characteristic of some nilpotent element $e \in \g_1$.\\
- Each assertion in Theorem \ref{Theorem.2.1} has its counterpart in the complex case \cite[Theorem 1]{Vinberg1979}. The converse  is not true. We do not have an analogue of Theorem 1(4) in \cite{Vinberg1979}, since
$e$  is not defined uniquely by $h$ up to $\Zz_{G_0} (e)$. This makes the classification  of  nilpotent elements
in Lie algebras over $\R$ more complicated than those over $\C$.
\end{remark}

We call a triple $(h,e,f)$ satisfying  the condition in Theorem \ref{Theorem.2.1}.i {\it an $\ssl_2$-triple}.  

\begin{proof} [Proof of Theorem \ref{Theorem.2.1}] i) Theorem \ref{Theorem.2.1}.i is obtained by combining the JMV theorem in \cite{Vinberg1979} for graded complex Lie algebras with a Jacobson's trick used in the  proof of \cite[Lemma 9.2.2]{C-M1993}.  By the JMV theorem   \cite[Theorem 1(1)]{Vinberg1979} there exists  a triple $(h _\R  + \sqrt{-1} h_\R ' \in \g _0 ^\C, e, f_\R + \sqrt{-1} f_\R' \in  \g _{-1}^\C)$  such that $h _\R, h_\R ', f_\R, f_\R '\in \g$ and
$$ [ h _\R, e] = 2 e, \, [e, f_\R ] = h _\R.$$

A Jacobson's trick \cite[proof of Lemma 9.2.2]{C-M1993}, provides us with an element  $z$ in  the centralizer $\Zz_\g ( e)$  of $e$ in   $\g$ such that
\begin{equation}
(ad_{h_\R} + 2) z = - [ h_\R,  f_\R ] - 2f_\R.
\label{2.2}
\end{equation}
(For  the convenience of the reader we recall that the existence of $z$ satisfying (\ref{2.2}) is obtained by showing the positivity of the eigenvalues
of $ad_{h_\R}$ acting on $\Zz_\g (e)$, hence the equation $(ad_{h_\R} +2) z = -[h_\R, f_\R] - f_\R$ has a solution
$z \in \Zz_\g (e)$ since $-[h_\R, f_\R] - f_\R \in  \Zz_\g (e) $.)
It is easy to see that we can  assume that $z \in \g _{-1}$. Then $(h_\R, e,  f_\R + z )$ satisfies  our condition in Theorem \ref{Theorem.2.1}.i. Any $h$ satisfying  the relation in Theorem \ref{Theorem.2.1}.i is   semisimple, since it is a semisimple  element in the Lie algebra $\ssl(2,\R) = \la e, f, h \ra_\R$. This proves Theorem \ref{Theorem.2.1}.i.

ii) There are two proofs  of this assertion. In the first proof we adapt the argument  in \cite[the proof of Theorem 3.4.10]{C-M1993},(Theorem of Kostant), which has been generalized in Theorem 1(2)
in \cite{Vinberg1979} for graded Lie algebras. Their proof, based on the  $\ssl_2$-theory,   works also for field $\R$. Let us explain their argument adapted to our case. Denote  by  $\Zz_{\g_0} (e)$ the centralizer of $e$ in  $\g_0$. 
 
If $h'$ is another element
satisfying  the condition in Theorem \ref{Theorem.2.1}.i, then  $h -h ' \in \Zz_{\g_0} (e)$.  The relations
in Theorem \ref{Theorem.2.1}.i imply that $h -h ' \in [\g_{ -1}, e]$.  Set
$\u_{\g _0} (e): = \Zz_{\g_0} (e) \cap [\g_{-1}, e]$.  Then $h'-h\in \u_{\g _0} (e)$.

Next, we claim that $\u_{\g_0}(e)$ is  an $ad_h$-invariant nilpotent ideal  of
$\Zz_{\g_0}(e)$. To prove this  claim we use   Lemma 3.4.5  in \cite{C-M1993}.

\begin{lemma}  \cite[Lemma 3.4.5]{C-M1993} Let  $e$ be a  nonzero nilpotent element of a  semisimple Lie algebra $\g$.
Then $\u _\g (e) : = \Zz_\g (e) \cap  [\g, e]$  is an $ad_h$-invariant  nilpotent ideal of $\Zz_\g (e)$.
\end{lemma}

To obtain  our claim from \cite[Lemma 3.4.5]{C-M1993} we observe that, if
a $\Z_m$-graded  ideal is nilpotent then its  subalgebra consisting of  homogeneous elements of zero degree is a nilpotent ideal  in the  subalgebra $\g_0$.

Set $U_0(e) := \exp  \u_{\g_0} (e)\subset \Zz_{G_0} (e)$.  
\begin{lemma}\label{347cm}  We have
\begin{equation}
Ad_{U_0(e)}(h) =h + \u_{\g_0} (e).
\end{equation}
\end{lemma}

We note that  Lemma \ref{347cm} is  a version of Lemma 3.4.7 in \cite{C-M1993}  due to Kostant.
\begin{proof}[Proof of Lemma \ref{347cm}]
The proof of Lemma 3.4.7 in \cite{C-M1993} carries to  our case easily, since   $\u_{\g_0}(e)$ is  $ad_h$-invariant.
Set
$$\u(e)_k : = \{ x\in  u_{\g_0} (e)|\, [h, x] = k  x\}.$$
Using the $\ssl_2$ theory we have following  decomposition
$$\u_{\g _0} (e)= \oplus _{ i = 1} ^ n  \u  (e)_k  $$
for some finite positive integer $n$.

To prove  Lemma \ref{347cm} it suffices to find  an element  $z\in \u_{\g _0} (e)$ for a given  $v \in \u_{\g _0} (e)$ such that
$Ad_{\exp  z} (h)  =h + v$.
We   approximate $z$ by $z_j$ inductively such that
\begin{equation}
z_j \in \oplus _{1\le i \le j}  \u(e)_i,  \label{a}
\end{equation}
  and 
\begin{equation}
Ad_{\exp  z_j} h  - ( h + v) \in  \oplus _{j +1\le i \le m} \u (e) _i.
\label{b}
\end{equation}
Set
$$z_{j +1} ' : = \text{ the component  of } ( Ad_{\exp_{z_j} } h -  ( h + v) )  \text { in } \u(e)_{j+1}.$$
Let
$$z_{j +1} = z_j +{ 1\over j +1} z'_{j+1} \in \oplus _{ 1\le i \le j +1}  \u (e)_i.$$
Then we check  immediately that    properties (\ref{a}) and (\ref{b})  carry over to $z_{j +1}$. Thus if we begin with
$z_1: = - v_1$, where $v_1$ is the component of $v$ is $\u(e)_1$, and setting $z: = z_n$, we get 
$Ad_{\exp _z} (h) = v$, as desired.
This proves  Lemma \ref{347cm}.
\end{proof}

Noe let us complete the proof of  Theorem \ref{Theorem.2.1}.ii.  We need to show the uniqueness of $h$  up to conjugacy
via an element  in $\Zz_{G_0} (e)$.  Suppose the opposite, i.e. there are two  $\ssl_2$-triples $(f, h, e)$ and
$(f', h, e')$ satisfying  the condition of Theorem \ref{Theorem.2.1}.i.  Then $h- h' \in\u_{\g_0} (e)$ as we have 
observed above.  By Lemma  \ref{347cm} there is an element  $x \in U_0 (e)\subset \Zz_{G_0}(e)$ such that
$\exp _x (h) = h '$ and $\exp _x (e) = e$.  This implies $\exp _x (f) = f'$. This  proves Theorem \ref{Theorem.2.1}.ii.

The second  proof  of  Theorem \ref{Theorem.2.1}.ii   uses  the  Vinberg argument in \cite[proof of Theorem 1 (2)]{Vinberg1979}. 
The first and the second proofs are distinguished by   different methods to prove  Lemma \ref{347cm}. In the second proof   the main  point  is to show that  the orbit $Ad_{U_0(e)} (h)$ is open  and closed in $ h + \u_{\g_0}(e)$. We   remark that  the   closedness of the orbit
$Ad_{U_0 (e) }(h) $   holds, since  this orbit is a component of  the  intersection of  $\g_1$ with the complexified  orbit $Ad_{U_0^\C (e) } h$, which is closed  by  \cite[proof of Theorem 1(2)]{Vinberg1979}.  The openess of  the orbit also holds, since $[h, \u_{\g_0}(e)] = \u_{\g_0}(e)$, which is a  consequence  of the identity $[h, \u_{\g_0 ^ \C} (e)] =
\u _{\g _0 ^\C} (e)$ proved by Vinberg in \cite{Vinberg1979}.

iii) Theorem \ref{Theorem.2.1}.iii  follows from the uniqueness of an $\ssl_2$-triple in a complex Lie algebra, see e.g  \cite[Lemma 3.4.4]{C-M1993}, or \cite[Theorem 1(3)]{Vinberg1979}.
\end{proof}

\medskip

Thanks to the JMV theorem we can characterize semisimple elements and nilpotent elements in $\g_1$
using the geometry of their $Ad_{G_0}$-orbits. 

\begin{lemma}\label{Lemma.2.3}   Element $x \in  \g_1$ is nilpotent if   and only if the closure of its orbit
 $Ad _{ G_0}(x)$  contains zero.   Element $ x\in \g_1$ is semisimple if and only if  its orbit  $Ad _{ G _0}(x)$
is closed. 
\end{lemma}

\begin{proof} Suppose that $x\in \g _1$ is nilpotent.
 By Theorem \ref{Theorem.2.1},   there is an element $h \in  \g _0 $ such that $[ h, x] = x$.  Clearly $\lim _{ t \to -\infty} Ad _{ \exp ( t\cdot h)}  (x) = 0$.  This proves the ``only if" part of the first assertion of
 Lemma \ref{Lemma.2.3}.

Now we suppose that  the closure of the orbit $Ad_{G_0} (x)$ contains zero. 
 Then  the  orbit  $Ad_{ \rho( G_0)} (x)$
contains zero, in particular $Ad _{ G _0 ^\C}(x)$ contains zero. By  \cite[Proposition 1]{Vinberg1976}, 
$x$ is a nilpotent element in $\g _ 1 ^\C$. Hence $x$ is a nilpotent element in $\g_1$. This proves the ``if" part
of the first assertion.

Let us  prove the second assertion of Lemma \ref{Lemma.2.3}. If $x$ is not semisimple, let us consider its Jordan decomposition  $x = x_s+ x_n$. The proof of  \cite[Proposition 3]{Vinberg1976}  yielda the existence of an element $l $ in the centralizer 
$\Zz _{\g^\C} (x_s)$ such that $[l, x_n] = x_n$. Writing $l = l' + l''$ where $l' \in \g_0$ and $l'' \in  \ \sqrt{-1} \g_0$,  we find that $[l' , x_n] = x_n$.
Then $\lim _{t \to - \infty} Ad _{\exp t l'} (x_n) = x_s$.   Hence the orbit
$Ad_{ G_0} (x)$ is not closed. This proves the ``if" part of the second assertion.

Now  assume that $x$ is semisimple. Then  the  orbit $Ad _{ G ^\C _0 }  (x)$ in $ \g _1 ^\C$ is  closed.
Hence the intersection  of this orbit   with  $\g_1 \subset \g_1 ^\C$ is closed  in $\g_1$.
Note that this intersection is a  disjoint  union of   $Ad _{ G_0} $-orbits  of  elements in $\g _1$. Since  each  orbit $Ad_{G_0} (y')$
is a submanifold  in $\g_1$,  it follows that   each  $Ad_{G_0}$-orbit in  this intersection is also closed. This proves the ``only if" part
of the second assertion.
\end{proof}

We adopt the following definition in \cite{Vinberg1976}.  Let $\g =  \oplus_{i =1}^m \g_i$ be  a   $\Z_m$-graded
semisimple  Lie algebra.
{\it A Cartan subspace}  in $\g_1$  (resp. $\g_1 ^\C$)  is a maximal subspace  in $\g_1$ (resp. in $\g_1 ^\C$) consisting of   commuting  semisimple  elements.  The classification of Cartan subspaces in $\g_1$ is well-known for
$m \le 2$, see \cite{Kostant1955}, \cite{Sugiura1959}, \cite{OM1980},  and unknown for $m \ge 3$.

\section{R-compatible Cartan involutions}

In this section we show  the existence of a Cartan involution  of a real $\Z_m$-graded   semisimple Lie algebra $\g$ which reverses the $\Z_m$-gradation on $\g$, see Theorem \ref{Theorem.3.7}. As a consequence,  there
is a 1-1 correspondence between $Ad_{G_0^ \C}$-orbits (resp. $Ad_{G_0}$-orbits)  on $\g_i ^\C$ and
$\g_{-i} ^ \C$, (resp.  on $\g_i$ and $ \g_{-i}$), see Corollary \ref{rev}.  
 We  also give important examples  of  real $\Z_m$-graded  semisimple Lie algebras.

Let  $\g = \oplus _{i =0} ^{m-1} \g_i$  be a $\Z_m$-graded semisimple Lie algebra and $\theta^ \C$  the automorphism of $\g ^\C$ associated with this  induced gradation.    It is easy to check that
\begin{equation}
 \tau _\g \theta ^ \C = (\theta ^\C ) ^ {-1} \tau _\g .
 \label{comp}
 \end{equation}
Since $\tau _\g ^ 2 = Id$,  (\ref{comp})  holds if and only if 
\begin{equation}
 \tau _\g (\theta ^ \C) ^{-1} = (\theta ^\C ) \tau _\g.
\label{comp2}
\end{equation}
Now let $\g$ be a real form in $\g ^ \C$ with a  $\Z_m$-gradation generated by $\theta ^\C$. If $\g$ satisfies
the relation (\ref{comp}), then   for any $ x\in \g^\C_ k$
$$\theta ^\C (\tau_\g ( x)) = \tau_{\g} (\theta  ^ \C) ^ {-1} (x) = \tau _\g ( \exp ^{\frac{- 2\pi \sqrt{-1} k }{m}} x) = \exp ^{\frac{2\pi \sqrt{-1}k }{m}} \tau _\g (x)  .$$
Hence  $\tau _\g ( \g_k ^ \C)  = \g _k ^ \C$, and therefore
\begin{equation}
 \g = \oplus _i ( \g \cap \g_i ^ \C).
\label{grad}
\end{equation}
 Thus  we  say that a real form  $\g $ of $\g ^ \C$ is {\it compatible  with $\theta^\C$}, if  (\ref{comp}) holds. Equivalently (\ref{comp2}) holds, and equivalently (\ref{grad}) holds.
 
\begin{remark}\label{inv} If $m \not= 2$, any real form $\g$  compatible with $\theta ^ \C$  is not invariant
under $\theta ^ \C$ unless  $m$ is even and  the only nonzero components  of $\g$ have  degree $0$ or
$m/2$. A real form $\g$   is invariant under $\theta ^\C$, if and only if  $\tau_\g $ commutes
with $\theta  ^\C$.
\end{remark}

Let $\u$  be  a compact real form of $\g^\C$ which is compatible with $\g$, i.e.  $\tau_\g \tau _\u = \tau _\u \tau _\g$.  Then $\g = \k \oplus \pg$ where
$\k = \g \cap \u$ and $\pg = \g \cap i \u$. The restriction of $\tau _u$ to $\g$ is a Cartan involution of
$\g$, which we also denote by $\tau _\u$, if  no misunderstanding  arises. 

\begin{definition}\label{Definition.3.2} A    Cartan involution $\tau_\u$ of a real $\Z_m$-graded semisimple   Lie algebra $\g =\oplus_{i=1}^m \g _i$ is called {\it R-compatible} with 
 the $\Z_m$-gradation, if $\u$ is invariant  under  the automorphism $\theta^\C$ associated with this gradation: 
 $\tau _\u \theta ^ \C = (\theta ^ \C)  \tau _\u$.
 \end{definition}
  
Clearly,  $\tau_\u$ is $R$-compatible with the $\Z_m$-gradation, if and only if $\tau_\u$ reverses the gradation
on $\g$ : $\tau _\u ( \g _k) = \g _{-k}$.   That explains   our use of the notion  of  a $R$-compatible involution.  

\begin{example}\label{Example.3.3} i)  Any  real $\Z_2$-graded semisimple Lie algebra $\g = \g_0 \oplus \g_1$  has a R-compatible Cartan 
involution,  see  \cite{Berger1957}, Lemma 10.2.  The classification  of all
$\Z_2$-graded  simple Lie algebras  has been given in \cite{Berger1957}.

ii) Let $x\in \g_1$. Let $\Zz_\g (x)$ be  the centralizer of $x$ in $\g$.  Clearly,  its complexification
$\Zz_{\g ^\C} (x)$ is invariant under  the action of $\theta ^\C$. Hence $\Zz_\g (x)$ inherits the $\Z_m$-grading, and  the commutant
$\Zz _\g (x)'$ of $\Zz_\g (x)$  is also a  real $\Z_m$-graded semisimple Lie
algebra. If $m=2$ and $x\in \g_1 \cap \pg$ or $x\in \g_1 \cap \k$,   the Cartan compatible involution $\tau_\u$ also preserves  $\Zz_{\g} (x)$.

iii) If $(\g, \tau_\u)$ and $(\g ', \tau _{\u'})$ are   real $\Z_m$-graded semisimple Lie algebras with
R-compatible Cartan involutions $\tau _\u $ and $\tau_{\u ' }$, then
their direct sum $\g \oplus \g '$ is also a real $\Z_m$-graded semisimple Lie algebra equipped with  the R-compatible
Cartan involution $\tau _{\u \oplus \u '}$.  Conversely, if $m$ is prime any   real $\Z_m$-graded semisimple Lie algebra is a direct sum of real $\Z_m$-graded simple Lie algebras (see \cite{Vinberg1976} for a similar assertion over $\C$, which implies 
our assertion).

iv) Let us consider the split algebra $\g = \e_{7(7)}$ - a  normal real form of the
complex Lie algebra $\e_7$. The  complex algebra $\g ^\C = \e_7 $ has the following   root system\\
$\Sigma=\{ \eps _i-\eps _j,  \eps_p + \eps _q + \eps _r + \eps _s, |i \not = j, \, ( p,  q, r, s \text {  distinct}), \sum _{i =1} ^8 \eps _i = 0\}$.\\
For the purpose of  fixing notations we recall  the following root  decomposition  of a complex
semisimple Lie algebra $\g^\C$  and its compact real form $\u$, see e.g. \cite[Theorem 4.2]{Helgason1978} and
\cite[Theorem 6.3]{Helgason1978}. Let us choose a  
Cartan subalgebra $\h_0^\C$  of $\g ^\C$. Denote by $E _\alpha$, $\alpha \in \Sigma$, the   corresponding  root  vectors such that $[E_\alpha ,E_{-\alpha}] =
{2 H_\alpha \over \alpha ( H_\alpha)}\in \h_0 ^\C$, see e.g. \cite{Helgason1978}, p.258.
We  decompose $\g$ as
\begin{equation}
\g^\C = \oplus _{\alpha \in \Sigma} \la H_\alpha \ra _\R\oplus _{\alpha\in \Sigma} \la E_\alpha\ra _\R \oplus _{\alpha\in \Sigma} \la E _{ - \alpha}\ra _\R .
\label{3.4}
\end{equation}
$\g ^\C$ has the following compact form $\u$, which is compatible  with $\g$:\begin{equation}
 \u = \oplus _{\alpha \in \Sigma} \la iH_\alpha\ra _\R \oplus_{\alpha\in \Sigma} \la i( E_\alpha + E _{ - \alpha})\ra _\R \oplus_{\alpha\in \Sigma}  \la ( E _\alpha - E_{-\alpha} )\ra _\R.
 \label{3.5}
 \end{equation}

Let $\theta ^\C$ be  the  involution of $\e_7$ defined in \cite{Antonyan1981} as follows\begin{equation}
{\theta^\C} _{| \h_0 } = Id,
\label{3.6.1}
\end{equation}
\begin{equation}
{\theta^\C}  ( E _\alpha ) = E_\alpha, \text { if }  \alpha  = \eps _i - \eps _j ,\\
\label{3.6.2}
\end{equation}
\begin{equation}
{\theta ^\C} ( E_\alpha ) =  - E _\alpha , \text { if } \alpha = \eps _i + \eps _j + \eps _k + \eps _l.
\label{3.6.3}
\end{equation}
Then $\theta ^\C (\g) = \g$, and $\theta ^\C (\u) = \u$. Hence $\theta^\C$  commutes with $\tau_\g$ as well as with $\tau _\u$.  Denote  by $\theta$ the restriction of $\theta ^\C$ to $\g$. Automorphism $\theta$  defines a $\Z_2$-gradation:
 $\g = \g _0 \oplus \g_1$, where $\g_0 = \ssl (8, \R)$. Clearly $\tau_\u$ is a R-compatible  with this $\Z_2$-gradation.
In \cite{Antonyan1981}  Antonyan proved that the space  $g-1 ^\C$ is linearly isomorphic to the space
$\Lambda ^4 (\C^8)$ of 4-vectors on $\C^8$.  Let $G_0 ^\C \subset  E^\C _7$ be the connected Lie subgroup with the Lie algebra $\g_0 ^\C$. Antonyan showed that the adjoint action of
$G_0 ^\C$ on $\g_1 ^\C$ is  exactly  the canonical action of $SL (8, \C)$ on the space  $\Lambda ^4 (\C^8)$.
 
v) Let us consider  a  real $\Z_3$-graded simple   Lie algebra  $\e _{8 (8)}$ which is a normal form
of  the complex algebra $\e_8$.  The root system $\Sigma$  of $\e_8$ is
$$\Sigma = \{ \eps_i -\eps_j, \pm (\eps_i + \eps_j+ \eps_k)\},\, (i,j,k \text { distinct}), \sum _{i = 1} ^9\eps_i  = 0\}.$$
In \cite{EV1978}  Vinberg  and Elashvili proved that there is an automorphism $\theta ^\C$ of order 3 on $\e_8$ defined by the following formulas
$$ \theta^\C  _{|\la H_{\alpha}, E_{\alpha}, \:  \alpha  = \eps_i - \eps_j\ra _\C}  = Id, $$
$$\theta^\C _{|\la E_{\alpha}, \alpha = (\eps_i + \eps _j + \eps _k)\ra _\C } = \exp  (i2\pi/3) \cdot Id, $$
$$\theta^\C _{|\la  E_{\alpha}, \alpha = -(\eps_i + \eps _j + \eps _k)\ra _\C } = \exp  (-i2\pi/3) \cdot Id. $$
It is easy to see that  $\theta ^\C$ defines a $\Z_3$-grading  on $\e_8$ as well as on 
$e_{8(8)}$.  Namely, we have $\e_{8 (8)} = \g_0 \oplus \g_1\oplus \g_{-1}$  where
$$\g _0 = \la  H_{\alpha}, E_{\alpha}, \:  \alpha  = \eps_i - \eps_j \ra _\R, $$
$$\g_1 = \la  E_{\alpha}, \alpha = (\eps_i + \eps _j + \eps _k)\ra _\R ,$$
$$\g_{-1} = \la  E_{\alpha}, \alpha = -(\eps_i + \eps _j + \eps _k)\ra _\R .$$
The compact form $\u$ of $\e_8$ defined as in   (\ref{3.5}) is   R-compatible with
this $\Z_3$-grading of $\e_{8(8)}$.
In \cite{EV1978}  Vinberg and Elashvili  proved that  the  space $\g _1 ^\C$ is  linearly isomorphic to
the space $\Lambda^3 (\C ^ 9)$ of 3-vectors on $\C ^9$ and the space  $\g _{-1} ^\C$ is  linearly isomorphic to  the space $\Lambda ^3 (\C^9) ^ *$
of 3-forms on $\C ^9$. Let $G _0 ^\C\subset E_8 ^\C$ be the connected Lie  subgroup with the  Lie subalgebra $\g_0^\C$. 
Vinberg and Elashvili showed that the adjoint action of 
$G_0 ^C$
on $\g _1 ^\C$ (resp. $\g_{-1} ^\C $)  is   exactly the canonical action
of $SL (9, \C)$ on the space $\Lambda ^3 (\C ^9)$  (resp. $\Lambda ^3 (( \C ^9 )^* )$.   
\end{example}

The following Theorem is  an analogue of Theorem 7.1 in \cite{Helgason1978} for  real $\Z_m$-graded Lie  semisimple Lie algebras.
The case  $m = 2$ is well-known, see \cite{Berger1957}.

\begin{theorem}\label{Theorem.3.7}   Let $\u'$ be a  real compact form of $\g ^\C$, which is  invariant under $\theta ^\C$.\\
1) There  exists  an automorphism $\phi $ of $\g ^\C$, which commutes with $\theta ^\C$,   such that
$\u =\phi (\u')$ is invariant under  $\tau _\g$  and under   $\theta ^\C$.\\
2)  Any  real $\Z_m$-graded  semisimple  Lie algebra
has a  Cartan involution, which reverses the gradation.
\end{theorem}

\begin{proof} 1) Our arguments are similar to those in the proof of  \cite[Theorem 7.1]{Helgason1978}. Let $B$ denote the Killing form on $\g ^\C \times \g ^\C$. The  Hermitian form $B_{\u'}$ defined on $\g ^\C \times \g ^\C $ by
$$ B _{\u'} (X, Y) = -B(X ,\tau _{\u'}( Y) ) $$
is  strictly positive definite, since $\u'$ is compact. 
The composition $\tau_g\tau_{\u'}$ is an automorphism of $\g ^\C$, so it leaves
the Killing form invariant. Thus we have
\begin{equation}
B( \tau _\g \tau _{\u '}X, \tau _{\u '} Y) = B ( X, (\tau _\g \tau _{\u '})^{-1} \tau _{\u '} Y)
\label{hel1}
\end{equation}
Taking into account  $\tau _\g ^2 = \tau _{\u '} ^ 2 =  Id$ we get  from (\ref{hel1})
$$B( \tau _\g \tau _{\u '}X,  Y)  = B( X, {\tau _\g \tau _{\u '}}^{-1}\tau _{\u '} Y)= B( X, \tau _{\u '}(\tau _\g \tau _{\u '}) Y)= B_{\u '}( X, \tau _\g \tau _{\u '} Y).$$
 Hence  $( \tau _\g \tau _{\u'} ) ^2 $ is  positive self-adjoint   w.r.t. $B_{\u'}$,  moreover it
commutes  with $\theta ^\C$, because  $\tau _\g \theta ^ \C = (\theta ^\C)^{-1}\tau_\g$ and $\tau _{\u '}$ commutes with $\theta ^\C$. 
It follows that the automorphism $ \phi :=  [(\tau _\g \tau _{\u'} ) ^2 ]^ {1/4}$  commutes with $\theta ^\C$.
(To see it, we choose an  orthogonal basis  $(e_j)$ of $\g^ \C$  w.r.t. $B_{\u'}$ which are also  eigenvectors  with eigenvalues
$a_i >0$ of $ ( \tau _\g \tau _{\u'} ) ^2 $  for all $i$.
We note that $\theta ^ \C$ commutes with $( \tau _\g \tau _{\u'} ) ^2 $  if and only if $\theta (e_i)$ is also eigenvector  of $( \tau _\g \tau _{\u'} ) ^2 $ with value $a_i$ for all $i$. 
Clearly,  $(e_i)$  and $\theta^\C (e_i)$ are also eigenvectors  of $[( \tau _\g \tau _{\u'} ) ^2]^{1/4} $  with eigenvalue $(a_i) ^{1/4}$. Therefore $\theta ^\C$  commutes also with  $ [( \tau _\g \tau _{\u'} ) ^2]^{1/4} $.)
Hence $\phi (\u')$
is invariant under $\theta ^\C$. 

The  invariance  of $\phi (\u')$ under $\tau _\g$ has been shown  in the proof of  \cite[Theorem 7.1]{Helgason1978}.  (For the convenience of  the reader  we briefly recall the proof. The invariance of $\phi (\u')$  under $\tau _\g$ 
is equivalent to the identity
\begin{equation}
\tau _\g \tau _{\phi (\u')} = \tau _{\phi (\u')} \tau _\g.
\label{hel2}
\end{equation}
Using the relation
$$\tau _{\u'} (\tau_\g \tau _{\u'} ) \tau _{\u '} ^{-1} = (\tau _\g \tau _{\u'} ) ^{-1} $$
we  get
\begin{equation}
\tau _{\u'} \phi \tau _{\u '} ^{-1} = \phi  ^{-1}
\label{hel3}
\end{equation}
Note that $\tau _{\phi (\u')} = \phi \tau _{\u'} \phi ^{-1}$.  Using (\ref{hel3}) and  $\phi = [ (\tau _\g \tau _{u'} ) ^2 ] ^{1/4}$ we get easily that the LHS of (\ref{hel2}) is equal  to RHS of (\ref{hel2}) and equal to $Id$.) This proves the first assertion of Theorem \ref{Theorem.3.7}.

2) By Lemma 5.2, chapter X in \cite{Helgason1978}, p. 491,  there is a real compact form $\u'$ of $\g ^\C$ which is invariant under 
$\theta ^\C$. Taking into account  the first assertion of Theorem \ref{Theorem.3.7}, we prove the second assertion. 

Here is another  short proof  of the second assertion due to  Vinberg \cite{Vinberg2009}.  Let us consider the group  $G(\theta ^\C, \tau_\g)$ generated by
$\theta ^\C$ and  $\tau_\g$  acting on  the space   $G ^\C/ U$  of all compact real forms of $\g ^\C$.  This group is finite, since  $\tau_ g  \theta ^ \C  = (\theta ^\C)^{-1} \tau_\g$.  As
E. Cartan proved \cite{Cartan1928}, see also  \cite[Theorem 13.5, chapter I]{Helgason1978} for a modern treatment, any compact group of motions of a simply connected
symmetric space of non-positive curvature has a fixed point.  Is is known that  $G ^ \C/ U$ is a symmetric space
of noncompact type, hence it has  nonpositive curvature, \cite[chapter VI]{Helgason1978}. The fixed point of  $G(\theta ^\C, \tau_\g)$
is the required  compact form. 
\end{proof}

\begin{corollary}\label{rev} A  R-compatible  involution $\tau _\u$ gives an isomorphism between
$Ad_{G_0}$-orbits  in $\g_i$ and $\g_{-i}$. The  $\C$-linear extension $\tau_\u ^ \C (= \tau_\u \circ \tau _\g)$  of $\tau _\u$ gives
an isomorphism  between $Ad_{G_0 ^ \C}$-orbits in $\g_i^\C$ and  $\g_{-i} ^ \C$.
\end{corollary}

\begin{proof} Denote by $\hat \tau_\u ^ \C$ the  involutive automorphism on $ G  ^ \C$ whose differential is $\tau_\u ^ \C$. Since $\tau_\u^\C (\g_0) = \g_0$ and $\tau _\u ^\C  ( \g_0 ^ \C) = \g_0  ^ \C$ we get
$$ \hat \tau_\u ^ \C (G_0 ) = G_0, \quad \hat \tau_\u ^\C ( G_0 ^ \C) = G_0 ^ \C.$$
 For  any $v \in \g _0 ^ \C$
and $e \in \g_{i} ^\C$ we have  $\hat \tau_\u  ^ \C  (\exp v) = \exp (\tau ^\C _\u ( v) )$ and
$$ \tau_\u ^ \C ( Ad_ {\exp v} e) = Ad _{\exp (\tau _\u ^ \C ( v))}(\tau_\u ^\C ( e) ).$$
 Consequently 
 $$\tau _\u ( Ad_{G_0 } e ) =  Ad_{G_0 } ( \tau _u e), \quad  \tau _ \u ^\C ( Ad_{G_0 ^ \C} ( e)) = Ad_{G_0 ^ \C}(\tau _\u  ^ \C (e)).$$ 
 This proves   our corollary. 
\end{proof}

\section{Classification of homogeneous nilpotent elements}

To characterize the  set of orbits  of homogeneous nilpotent elements in a real $\Z_m$-graded  semisimple Lie algebra $\g$
is more complicated than  to characterize the  set of orbits  of nilpotent elements in the  case of complex $\Z_m$-graded  semisimple Lie algebras, since the orbit of  a nilpotent element
$e$ in  $\g$  is not defined uniquely by its characteristic.  If $ m = 1$, i.e. $\g$ is regarded without gradation, a complete  classification of nilpotent elements in $\g$ can be obtained
  using the Cayley transform and the  Vinberg  method of classification of nilpotent elements in  an associated complex $\Z_2$-graded
semisimple Lie algebra, see e.g. \cite{Djokovic1988}. We do not know how to generalize this method for   $m\ge 2$.
Our  method of characterization of the  set of orbits  of homogeneous nilpotent elements   in
a  real $\Z_m$-graded Lie algebra $\g$ is divided in the following steps. In Lemma 4.1 we prove that  there is an injective map from the set of  the $Ad_{G_0}$-conjugacy classes of characteristics in $\g$ to  the set of $Ad_{G_0^\C}$-conjugacy classes of  characteristics in $\g^\C$. Recall that a classification of   characteristics
in $\g ^\C$  can be obtained
by the Vinberg method of support \cite{Vinberg1979}.   In Remark \ref{Remark.5.2}, taking account the Djokovic classification of real forms of a complex $\Z$-graded  semisimple Lie algebra, we summarize these  results in an algorithm  to  classify characteristics in $\g$.
Then we show  in   Theorem \ref{Theorem.5.5} that  there is a 1-1 correspondence
between 
$Ad_{G_0}$-orbits of  nilpotent elements $e\in \g_1$ with a given  characteristic $h$
and the   set of open $\Zz_{G_0} (h)$-orbits  in $\g_1 (\frac{h}{2})$. This set  is closely related to  the set of connected components of a semialgebraic    set  in $\g_1 (\frac{h}{2})$.  
In Remark \ref{alg2} we  explain our  algorithm to count the number of  conjugacy classes of nilpotent elements in $\g_1$  as well
as to  choose a sample  representative for each  conjugacy class.  We note that this algorithm is highly complicated,  so we need  a  sufficient   computer power and a suitable  software package for interesting applications, see Remark \ref{rsa1}. In Example \ref{lin} we demonstrate  our algorithm  in a very simple case with a $\Z_2$-graded Lie algebra
$\g = \ssl_2 (\C)$  regarded as a Lie algebra over $\R$.
\medskip

Let $e$ be a nilpotent element in $\g_1$ and $h\in \g_0$ its characteristic. Then $h$ is also a  characteristic  of $e$
in $\g ^ \C$. A classification of $Ad_{G_0 ^ \C}$-conjugacy classes of characteristics in $\g_0 ^ \C$
can be obtained by using the support method of Vinberg in \cite{Vinberg1979}.  To define a support  of a nilpotent element $e\in \g_1 ^\C$  we choose  a Cartan subspace $\h$  in  the normalizer   $\Nn _{\g_ 0 ^ \C}(e)$  such that $\h \ni h$, where $h \in \g_0 ^\C$ is  a characteristic  of $e$. Let $\phi$ be   the character of $\h$  defined by
$$ [u, e] = \phi (u) (e) \text { for all } u\in \h  \text { and } $$
Set
$$\g(\h,\phi): = \bigoplus_k \g_k(\h, \phi), |\: \g_k(\h,\phi) = \{ x\in \g_{k \mod m}: [u, x] =  k \phi(u) x\: \forall u\in \h\}. $$
It is known that  $\g (\h, \phi)$ is a $\Z$-graded reductive Lie algebra \cite[Lemma 2]{Vinberg1979}.
Recall   that  a complex support $\s^\C(h)$ of $e$  is defined by
$$\s^\C(h) : = \g ' (\h,\phi)   -\text{ the commutant   of  }  \g (\h,\phi).$$
Clearly $\s ^\C (h)$ is defined by  $h$ uniquely up to conjugacy by elements  in $\Nn_{G_0 ^\C} (e)$.
Vinberg proved that  $\s^\C (h)$ is a locally flat $\Z$-graded semisimple Lie algebra in $\g^\C$ whose defining element is half of a characteristic $h$ of $e$ (``locally flat" means  $\dim \s_0 ( h) =\dim \s_1 (h)$)  \cite[\S 4]{Vinberg1979}.
We  define   a  real  support   $\s(h)$ of a nilpotent element $e$ in a real $\Z_m$-graded  semisimple Lie algebra  $\g$ in the same way. Here  we choose  $\h$ to be a  maximal 
$\R$-diagonalizable  Cartan subspace  in $\Nn_{\g_0}(e)$ containing $h$. Such a choice is unique up to a conjugacy by
 elements in $\Nn_{ G_0 } (e)$. Clearly,   the complexification of a real support of $e$ is a complex support of $e$ in 
$\g ^\C$.

It is known that the $Ad_{G_0 ^ \C}$-conjugacy classes of   characteristic elements  $h\in \g_0 ^ \C$
are in a 1-1 correspondence with the $Ad_{G_0 ^ \C}$-conjugacy classes of locally flat $\Z$-graded
semisimple Lie subalgebras $\s( h)$ in $\g^ \C$ \cite{Vinberg1979}. We refer the reader to \cite{Vinberg1979}
and \cite{Djokovic1982} for more details on $\Z$-graded  semisimple Lie algebras and $\Z$-graded locally flat   semisimple Lie algebras over $\C$ or over  $\R$.

\begin{lemma}\label{Lemma.5.1} i) There exists an injective map  from  the set of $Ad_{G_0}$-orbits of  characteristics in $\g$ to
the  set of $G_0 ^\C$-orbits of  characteristics  in $\g ^\C$.\\
ii) Let $h\in \g_0$ be a characteristic  of a nilpotent element in $\g_1$. Then  $Ad_{G_0^\C} (h) \cap \g_0  = Ad _{G_0}(h)$.
\end{lemma}

\begin{proof}  i) First we note that if $h\in \g$ is a  characteristic element  then it is  also a characteristic element in $\g ^\C$. 
Thus we have a map from the conjugacy classes of  characteristics in $\g$ to the conjugacy classes of characteristics in $\g ^\C$.
We  will show that this map is  injective.  Suppose  that $h_1,h_2\in \g_0$ are   characteristics  in $\g$
such that $Ad _X h_1 = h_2$  for $X \in G_0 ^\C$. Let $\tau_\u$ be a R-compatible Cartan involution in Theorem \ref{Theorem.3.7}. Note that the  restriction of $\tau_\u$ to $\g_0$ leaves the center  of $\g_0$ as well as the commutant $\g_0'$ of $\g_0$  invariant. Moreover the restriction of $\tau_\u$ to $\g_0 '$ is also a Cartan involution of $\g_0 '$. By the theory of Cartan subalgebras  in   real reductive Lie algebras, see. e.g. \cite[chapter IX, Corollary 4.2]{Helgason1978} we can assume that $h_1, h_2 \in  \Zz (\g_0) \oplus  \pg_0 '$, where  $\g_0' = \k_0 ' \oplus \pg_0 '$ is the  Cartan decomposition of $\g_0'$ with respect to $\tau_\u$.   By   \cite[Theorem 2.1]{Rothschild1972}, which asserts that  two semisimple elements in $\pg_0'$ are $G_0 ^\C$-conjugate if and only if they are  $G_0$-conjugate,
there exists $Y \in G_0 $ such that $Ad_Y h _1  = h_2$, since $h_1$ and $h_2$ are $G_0 ^\C$-conjugate.

ii) Clearly   Lemma \ref{Lemma.5.1}.ii is a consequence of Lemma \ref{Lemma.5.1}.i.

\end{proof}

\begin{remark}\label{Remark.5.2} Using  Lemma \ref{Lemma.5.1} we  obtain    a classification of conjugacy classes of characteristics
in $\g$ as follows. First we find   all complex  supports  in $\g^\C$  by  Vinberg method in \cite{Vinberg1979}.
There are only a finite number of them. Next, we find  the real forms  of  these  complex supports  using the  Djokovic classification
of real forms of complex  $\Z$-graded   semisimple Lie algebras  in \cite{Djokovic1982}.  In the third step we decide which
real  form of  a given complex support admits an embedding to $\g$  whose complexification is the given complex
support. This  step  can be done  using the theory of  representations of  real semisimple   Lie algebras, see e.g.
\cite{Iwahori1959}, \cite{Vinberg1988}.  
  Lemma \ref{Lemma.5.1} shows that in the third step
there exists  not more than one  real  form for each given  complex support. The  defining element of the corresponding
real support is   half of our desired characteristic.
\end{remark}

Now let us fix a characteristic  $h \in \g_0$ corresponding to a nilpotent element $e \in \g_1$.
Let us consider the following $\Z$-graded algebra
$$\g(\frac{h}{2}): = \bigoplus_k \g_k(\frac{h}{2}), |: \g_k(\frac{h}{2}) = \{ x\in \g_{k \mod m}: [\frac{h}{2}, x] =  k x\}.$$
 Clearly the centralizer $\Zz_{G_0} (h)$ of $h$ in $G_0$ acts
on $\g(\frac{h}{2})$ preserving  the $\Z$-gradation. The Lie  algebra of $\Zz_{G_0} (h)$  is
$\g_0 (\frac{h}{2})$. 
It is known \cite[proof of Theorem 1 (4)]{Vinberg1979} that $e \in \g_1(\frac{h}{2})$, moreover $[\g_0(\frac {h}{2}), e] = \g_1(\frac{h}{2})$.
Equivalently, $e$ belongs to an open  $Ad_{\Zz_{G_0} (h)}$-orbit    in $\g _1(\frac{h}{2})$.
An element $e\in \g_1$ (resp. $\g_1^\C$) is called {\it generic}, if   orbit $Ad_{\Zz_{G_0} (h)}(e)$  is open in $\g_1$, 
(resp. $Ad_{\Zz_{G_0 ^\C}(h)} (e)$ is open in $\g_1 ^\C$). Otherwise $e$ is called
{\it singular}.  By the definition  the genericity of an element $e\in \g_1$ implies the genericity of
any  element in the  orbit $Ad_{\Zz_{G_0 ^\C}(h)}(e)$.
The following Theorem \ref{Theorem.5.5} generalizes   a theorem  \cite[Theorem 6.1]{Djokovic1983}   due to Djokovic.

\begin{theorem}\label{Theorem.5.5} Let $(h, e, f)$ be a  $\ssl_2$-triple.  The inclusion
$\g_1(\frac{h}{2}) \to \g_1$ induces a bijection between  the open $Ad_{\Zz_{G_0} (h)}$-orbits
in $\g_1 (\frac{h}{2})$ and  the $Ad_{G_0}$-orbits contained in $Ad_{G_0 ^\C }(e) \cap \g_1$. 
\end{theorem}

\begin{proof} Suppose that $Ad_{\Zz_{G_0} (h)}(e')$ is an open orbit in $\g_1(\frac{h}{2})$.  then $e$ and $e'$ are generic elements in $\g-1 ^\C$, hence  $e'$ belongs to the  orbit $Ad_{\Zz_{G_0^\C}(h)  }(e)$ in $\g_1 ^\C$, (that is a remark due to Vinberg  in \cite[proof of Theorem 1(4)]{Vinberg1979}.  This  defines a map
from the set of  open $Ad_{\Zz_{G_0} (h)}$-orbits
in $\g_1 (\frac{h}{2})$ to the set of   $Ad_{G_0}$-orbits contained in $Ad_{G_0 ^\C }(e) \cap \g_1$.

We will show that this map is surjective.    Let $e'\in Ad_{G_0 ^\C} (e)\cap \g_1$. Let $h'\in \g_0$ be  a characteristic
of $e$. By the JMV theorem for the complex case, $h$ and $h'$ belong to  the same  $Ad_{G_0 ^\C}$-orbit.
Lemma \ref{Lemma.5.1}.ii implies that
there exists $X \in G_0 $ such that $Ad_X( h')  = h$.  Clearly $Ad_Xe'\in \g_1 (\frac{h}{2})$, since $[Ad_X (h'), Ad_X ( e') ]= Ad _X ( e')$. Element $Ad_X e' $ is generic  in $\g_1(\frac{h}{2})$, since it lies in the  orbit $Ad_{\Zz_{G_0^\C}(h)}(e)$.    This proves the surjectivity of the considered map.

It remains to show that this map is injective.   First we will prove the following

\begin{lemma}\label{Lemma.5.6} (cf. Lemma 6.4 in \cite{Djokovic1983}) {\it Let 
$e'$ be a generic element in $\g_1 (\frac{h}{2})$. Then  there exists $f'\in \g_{-1}(\frac{h}{2})$ such that $(h,e',f')$ is an $\ssl_2$-triple.}
\end{lemma}

\begin{proof}  Let $e$ and $e'$ be nilpotent elements satisfying the condition of Lemma \ref{Lemma.5.6}. Then $e$ and $e'$ are generic elements in $\g_1^\C$. 
 By a Vinberg remark in \cite[proof of Theorem 1.4]{Vinberg1979} there is an element $Y \in \Zz _{ G_0 ^\C} (h)$
such that $Ad_ Y ( e)  = e'$.  Clearly $(h, e', Ad_Y  (f))$ is a $\ssl_2 ^\C$-triple in $\g_1^\C $, moreover
$Ad_Y(f) \in \g_{-1}^\C(\frac{h}{2})$, since $f \in \g_{-1}(\frac{h}{2})$.  Since $h$ and $e'$ define
their $\ssl_2$-triple uniquely by Theorem \ref{Theorem.2.1}.iii, we get $Ad_Y (f) \in \g_{-1}^\C(\frac{h}{2})\cap \g_{-1} = \g_{-1}(\frac{h}{2})$.
\end{proof}

Let us complete the proof of  Theorem \ref{Theorem.5.5}. Suppose that $e$ and $e'$ are  generic elements
of $\g_1 (\frac{h}{2})$ such that $e'= Ad_X e$ for some $X\in G_0$. We will  show  that $e$ and $e'$ are in the same
open orbit of $\Zz_{G_0} (h)$. By Lemma \ref{Lemma.5.6}  there are elements $f$ and $f'$ in $\g_{-1}(\frac{h}{2})$ such that  $(h, e, f)$ and $(h, e', f')$
are $\ssl_2$-triples in $\g$. Note that $(Ad_ X h, e', Ad_X f)$ is a  $\ssl_2$-triple in $\g$.
By Theorem \ref{Theorem.2.1}.ii there exists an element $Y \in G_0$ such that $Ad_Y  (e') = e'$,
$Ad _Y ( Ad _X h) = h $ and $Ad_Y ( Ad _X f)  = f'$.  Thus $e' = Ad_{Y \cdot X } e$, where $Y\cdot X \in \Zz_{G_0} (h)$.  This proves the injectivity  of  our map.
\end{proof}

Now we proceed to   classify the open $\Zz_{G_0} (h)$-orbits
in $\g_1 (\frac{h}{2})$.

 Denote by $\g_i( \frac{h}{2})'$ the  $i$-th  component  of the commutant  of $\g(\frac{h}{2})$ which has
 the induced $\Z$-gradation from $\g(\frac{h}{2})$.  Since
$\g_1(\frac{h}{2}) = [\g_0 (\frac{h}{2}), \g_1(\frac{h}{2})]$,  we get
\begin{equation}
 \g_1(\frac{h}{2}) ' = \g_1 ( \frac{h}{2}).
 \label{g1}
\end{equation}
Since  $\Zz (\g (\frac{h}{2})) \subset \g_0 ( \frac {h }{2})$, we have $ \g_0 ( \frac {h}{2} ) = \Zz(\g(\frac{h}{2}))
\oplus \g_0 (\frac{h}{2}) '$. Hence
\begin{equation}
[\g_0(\frac{h}{2})' , g_1(\frac{h}{2})  ] =  g_1(\frac{h}{2}). 
\label{g01}
\end{equation}
Denote by $\Zz_{ G_0}(h)'$ the connected subgroup in $G_0$ whose Lie algebra is $\g_0(\frac{h}{2})'$.
An element $ e_i\in \g_i (\frac{h}{2})'$ is called {\it generic}, if the orbit
$Ad_{\Zz_{G_0}(h)'}(e_i)$ is open in $\g_i (\frac{h}{2})$. Equivalently, $ [\g_0 (\frac{h}{2})', e_i] = \g_i (\frac{h}{2})$. 

Let $\Zz_{G_0} (h) ^0$ be  the identity connected component of $\Zz_{G_0} (h)$.  From (\ref{g1}) and (\ref{g01}) we get immediately

\begin{lemma}\label{red} There exists a 1-1 correspondence between   the set of open   $Ad_{\Zz_{G_0} (h)^0}$-orbits  in $\g_1 (\frac{h}{2})$
and  the set of open $Ad_{\Zz_{G_0}(h)'}$-orbits in $\g_1  (\frac{h}{2})' = \g_1 (\frac{h}{2})$. 
\end{lemma}

\begin{remark}\label{nil} Clearly, all elements in $\g_i ^\C (\frac{h}{2})'$ are  nilpotent, if $ i \not = 0$.
Proposition 2 in \cite{Vinberg1976} asserts that there is only  a finite number  of $\Zz_{G_0^\C}(h)'$-conjugacy  classes
of  nilpotent elements in $\g_i^\C (\frac{h}{2})'$. Hence it follows that the  set of  generic nilpotent elements
in $\g_i ^ \C (\frac{h}{2})$ is open and dense in $\g_i^\C(\frac{h}{2}) '$. Since the number  of  $Ad_{\Zz_{G_0} ( h)} '$-orbits in  a $Ad_{\Zz_{G_0^\C}(h)'}$-orbit
 is  finite \cite{BCH1962}, Proposition 2.3, it follows that
for any  $i \not = 0$ the  set of  generic elements   in $\g_i(\frac{h}{2})' $ is  open and dense.
\end{remark}

Let us analyze the set of  open $Ad_{\Zz_{G_0}(h)'}$-orbits in $\g_1$.
Recall that an element $e$ in $\g_1(\frac{h}{2})$ (resp. in $\g_1^\C(\frac{h}{2})$)  is called {\it singular},   if it is not generic.
Equivalently
\begin{equation}
\dim [\g_0 (\frac{h}{2})',  e] \le \dim \g_1(\frac{h}{2}) -1.
\label{dim}
\end{equation}

Let $f_1, \cdots, f_m$ be a basis in $\g_0 (\frac{h}{2})'$.
Let us choose an basis  $e_1, \cdots, e_n$ in $\g_1 $. We write
$e = \sum _j  a_j ( e) e_j, \:  a_j \in \R$.
Then $[e, f_i] = \sum a_j(e) [ e_j, f_i] =\sum_{j, k} a_j(e) c_{ij} ^ k f_k$. Set $b_{ik} ( e): = \sum_j a_j (e)  c_{ij}  ^ k$.
Note that $e$ is singular, if and only if the matrix $(b_{ij} (e))_{i = 1, m} ^{j = 1, n}$ has
rank less than or equal to $n -1$. Note that $m \ge n$. 
Denote by $P_l$, $l = 1, \binom{n}{m},$ the sub-determinants of  $(b_{ij})$.
Clearly $e$ is singular, if and  only if $P_l ( e) = 0$ for  all $l$.  

\begin{lemma}\label{rsa} There is a 1-1 correspondence between  the set of open $Ad_{\Zz_{G_0}(h)^0}$-orbits in $\g_1(\frac{h}{2})$
and the  set of connected components of the  semialgebraic set  $\{ x\in \g_1(\frac{h}{2})|\: \sum _{l=1}^{\binom{n}{m}} P_l^ 2 (x) > 0\}$. The number
of open $Ad_{\Zz_{G_0}(h)^0}$-orbits in $\g_1(\frac{h}{2})$ is finite.
\end{lemma}

\begin{proof} The first assertion  follows from Lemma \ref{red} and our consideration above. The second assertion follows from the first one.
\end{proof}

\begin{remark} \label{rsa1} In \cite[chapter 16, Theorem 16.14]{Basu} the authors  offer  an algorithm to compute the number of  the connected components  of a semisalgebraic set and produce sample  representative for   each connected component. Their algorithm also allows to recognize, whether  given two points in a semialgebraic set belong to the same connected component of this set. This algorithm is highly
complicated and  we hope to  implement it   in  future  using  an appropriate software package.
\end{remark}

It remains to  consider  whether two given open  connected $Ad_{\Zz_{G_0}(h)^0}$-orbits in $\g_1(\frac{h}{2})'$ belong  to the same 
$Ad_{\Zz_{G_0}(h)}$-orbit in $\g_1 (\frac{h}{2})$.  Let $e_i$, $i= \overline{1, M}$, be representatives of the   connected open $Ad_{\Zz_{G_0}(h)^0}$-orbits
in $\g_1(\frac{h}{2})$ obtained by  the algorithm in \cite{Basu}, see Remark \ref{rsa1}.  Since  the  group $\Zz_{Ad_{G_0}}(h)$ is connected  \cite[Lemma 5]{Kostant1963},   the group $Ad_{\Zz_{G_0} (h)}$ is generated by $Ad_{\Zz_{G_0}(h)^ 0}$ and  the  subgroup $Ad_{\Zz (G_0)}$ acting $\g_1$.  Denote by
$F(e_k)$ the set of all elements  $X \in \Zz(G_0)$  such that $Ad_X (e_k)$ belongs to the orbit  $Ad_{\Zz_{G_0}(h)^0}(e_k)$. Clearly $F(e_k)$ is  a subgroup of
$\Zz(G_0)$.

\begin{lemma}\label{alg1} The quotient $\Zz (G_0)/F (e_k)$ is a finite  abelian group. There exists an algorithm to find  representatives  $Y_{ k,i}, i = \overline{1, N}$, of
the  coset $\Zz (G_0)/F (e_k)$ in $\Zz(G_0)$.
The orbit $Ad_{\Zz_{G_0} (h)}(e_k)$  is a  disjoint union of $N$ open connected orbits  $Ad_{\Zz_{G_0}(h) ^ 0}   (Y_{k,i} (e_k))$, $i = \overline{1, N}$.
\end{lemma}

\begin{proof} We know that $\Zz (G_0)$ is a finitely generated abelian group, which can be find explicitly \cite{Vinberg1988}.  Let $X_1, \cdots ,X_l$ be generators
of $\Zz(G_0)$.  Since  the number  of  connected  open $Ad_{\Zz_{G_0} (h) ^ 0}$-orbits  in $\g_1 (\frac{h}{2})$  is finite,  for each $j \in \overline{1,l}$ there exists a finite number  $p(j)$ such that $ Ad_{ X _j ^ {p(j)}} (e_k)$  belongs to the orbit $Ad_{\Zz_{G_0}(h) ^ 0}   ( (e_k))$.
This proves the first assertion of Lemma \ref{alg1}. The second assertion  follows from the proof of the first assertion using the  algorithm in
\cite{Basu}, see Remark \ref{rsa1}. The last assertion follows from the second assertion.
\end{proof}

\begin{remark}\label{alg2} We summarize our result  in the following algorithm to find   conjugacy classes of nilpotent elements
of degree 1  in a real  $\Z_m$-graded  semisimple Lie algebra $\g$. First  we classify    characteristics of nilpotent elements   in $\g_1$ using
the algorithm in Remark \ref{Remark.5.2}.   Theorem \ref{Theorem.5.5} shows that the conjugacy classes of nilpotent elements in $\g_1$ having a  given  characteristic
$h$ is  in a 1-1 correspondence  with  the set of  open   $Ad_{\Zz_{G_0} }(h)$-orbits in $\g_1 (\frac{h}{2})$. Using Lemma \ref{rsa} and Lemma \ref{alg1}
we compute the  number of    open   $Ad_{\Zz_{G_0}} (h)$-orbits in $\g_1 (\frac{h}{2})$ as well as  choose  sample representatives  for each open orbit  with help of the  algorithm in \cite{Basu}, see also  Remark \ref{rsa1}.
\end{remark}

\begin{example}\label{lin} Let us consider one very simple example to show how our algorithm works.
Let $\g = \ssl_2 (\C)$ be a simple Lie algebra over $\C$ and $\g _0 = \ssl_2 (\R)$ its Lie subalgebra. Then
$\g_0$ is the fixed point set of the involution $\theta$on $\g$ defined by $\theta(x) = \bar x$.
We write  $\g = \g_0 + \g_1$, where $\g_1 = \sqrt{-1}\g_0 \subset \ssl_2 (\C)$. The adjoint action of $G_0 = SL(2,\R)$ on
$g_1$ is equivalent to the adjoint action of $G_0$ on $\ssl_2 (\R)$. Clearly, $g ^\C = \ssl_2(\C) + \ssl_2 (\C)$, and $\g_0 ^\C = \ssl_2 (\C)$. It is known that there is only a unique nilpotent  $Ad_{G_0 ^\C}$-orbit in $\g_1 ^\C$, whose characteristic is conjugate to $h = diag (1, -1)\in \g_0^\C$. By Lemma \ref{Lemma.5.1}  the element $h$
is also the unique  (up to conjugacy) characteristic element in $\g_0$.  Let us consider the $\Z$-graded Lie algebra
$\g(\frac{h}{2})$.  We have
$$ \g (\frac{h}{2}) = \g_{-1} (\frac{h}{2})+ \g_0 (\frac{h}{2})+\g_1 (\frac{h}{2})$$
where
$$\g_0 (\frac{h}{2})= \la h\ra _\R ,$$
$$\g_1 (\frac{h}{2})= \ra \left ( \begin{array}{cc}0 & \sqrt{-1}\\
0 & 0
\end{array}) \right \ra _\R \subset \g_1,$$
$$\g_{-1} (\frac{h}{2})= \ra \left ( \begin{array}{cc}0 &  0 \\
\sqrt{-1} & 0
\end{array}) \right \ra _\R \subset \g_1.$$
By Theorem \ref{Theorem.5.5}  there is a 1-1 correspondence between the conjugacy classes of nilpotent elements
in $\g_1$  and open $Ad_{\Zz_{G_0} (h)}$-orbits in $\g_1(\frac{h}{2})$. Since
$\Zz(G_0) = Id$. by Lemma 4.7 there is a 1-1 correspondence  between  the latter orbits and the connected
components  of the semialgebraic  set $\{ x^ 2 > 0\}$ in $\g_1 (\frac{h}{2}) = \R$. Hence there are exactly
two conjugacy classes  of nilpotent elements in $\g_1$.
\end{example}

\section{Orbits in a real $\Z_2$-graded semisimple Lie algebra}

In this section, using  results in the previous sections, we  describe the set of  homogeneous elements in
a real $\Z_2$-graded Lie algebra $\g$, see Remark \ref{sumix} for a summarization.

 The restriction  to   real $\Z_2$-graded semisimple Lie algebras is motivated by the
fact that   we do not have a classification of Cartan subspaces in $\g_1$, if $ m \ge 3$.  A classification
of Cartan  subspaces  in $\g_1$ in a $\Z_2$-graded real semisimple Lie algebra has been given by  Matsuki and Oshima \cite{OM1980}, based on  an earlier work
by Matsuki \cite{Matsuki1979}. 
 
Let us first consider  the class of  semisimple elements in $\g_1$. Any
semisimple element in $\g_1$ belongs to a Cartan subspace in $\g_1$.

\begin{lemma}[\cite{OM1980}]\label{stan} Let $\tau_\u$ be a R-compatible Cartan involution of a
real $\Z_2$-graded semisimple Lie algebra $\g$. Every Cartan subspace $\h \subset \g_1$ is  $Ad_{G_0}$-conjugate  to a Cartan subspace
$\h_{st}$ in $\g_1$
which is  invariant under the action of $\tau_\u$.
\end{lemma}

A Cartan subspace $\h_{st}$ in $\g_1$  which is invariant under   the  action of $\tau_\u$ is called  {\it a standard Cartan subspace}.  It is known that there are only finite number of standard Cartan subspaces,  moreover there is algorithm to   find them \cite{OM1980}. Let $\g  = \k \oplus \pg $ be the Cartan decomposition of $\g$ w.r.t. $\tau _\u$. Then
$\h_{st} = (\h_{st} \cap \k) \oplus (\h_{st} \cap \pg)$. Denote by $K_0$ the  connected Lie subgroup in $G_0$  with  Lie algebra $\k$.

\begin{proposition}\label{c1}  Suppose that $h , h ' \in \h_{st}$ are $Ad_{G_0}$-conjugate.
Then they are $Ad_{K_0}$-conjugate.
\end{proposition}
\begin{proof}  We employ   ideas  in \cite{Rothschild1972} for our proof. Let $ h = h_\k + h _\pg$  and $h' =
h_\k' + h _\pg'$  be the decomposition
of $h$  and $h '$ into  elliptic and vector parts. Suppose
that $h = Ad_X (h')$, where $X \in G_0$.   Since $Ad_X$ does not change the eigenvalues, $h_{\pg} = Ad_X(h _{\pg}')$.
Suppose that $h_{\pg} \not = 0$. We note that $G_0 = \exp (\g_0 \cap \pg) \cdot K_0  $, and
  $\exp (\g _0 \cap \pg) \subset \exp \sqrt{-1} \u_0$.  Now suppose that $ X = A \cdot Y$
where $Y \in K_0$ and $ A \in \exp i\u_0$.  Let $y = Ad_Y h_{\pg}\in \sqrt{-1}\u_1$.  Then
$(Ad_A) \sqrt{-1}y =\sqrt{-1} h_{\pg}' = \tau_\u (Ad _A \sqrt{-1} y) = Ad_A ^{-1}\sqrt{-1} y$, so $ Ad_A ^2 y = y$. If $A \not = Id$ this implies that
$Ad_A$ has at least  one  eigenvalue $(-1)$, which contradicts the fact that
$Ad_A$  is a positive definite  transformation.
 
Hence $A = Id$ and $X = Y\in K_0 \subset G_0$.    This proves the first assertion, if $h_{\pg} \not = 0$. If $h_{\pg} = 0$ then
$h_\k \not = 0$ and we can   apply the same argument to  conclude that $X \in K_0$.
\end{proof}

Since any  semisimple element in $\g_1$ is $Ad_{G_0}$-conjugate to an element in some standard Cartan subspace
in $\g_1$, using the Cartan theory  of symmetric spaces, see e.g.\cite{Helgason1978}, we get

\begin{corollary}\label{easy} The set of $Ad_{G_0}$-conjugacy classes of semisimple elements in $\g_1$ with pure imaginary  or zero  eigenvalues (elliptic   semisimple elements)  coincides with  the  quotient set of  a   Cartan subspace (maximal abelian subspace)  $\h_{1\k} \subset (\g_1 \cap \k)$ under
the action of  the Weyl group  of the $\Z_2$-graded  symmetric Lie algebra $ \k_0\oplus \k \cap \g _1$.  The  set of $Ad_{G_0}$-conjugacy classes of
real semisimple elements  in $\g_1$  coincides with the quotient set of a  Cartan subspace (maximal abelian subspace)  $\h_{1\pg} \subset (\g_1 \cap \pg)$
under the action of the  Weyl group of the $\Z_2$-graded symmetric Lie algebra $\k_0 \oplus \g_1 \cap \pg$.
\end{corollary}

By Corollary \ref{easy}  $h_\k$ is   conjugate to some element in  a  Cartan subspace $\h_{1\k}\subset \g_1\cap \pg$.   Thus to classify all
semisimple elements in $\g_1$ it suffices to  classify all semisimple elements in $\g_1$ whose  elliptic part  is an element in $\h_{1\k}$.

\begin{corollary}\label{mix}  The  set of $Ad_{G_0}$-equivalent  elements $h$ with
given   elliptic part $h_\k \in \h_{1\k}$  coincides with the  quotient set of  a Cartan subspace in $\Zz_{\g_1 \cap \pg} (h_\k)$
under the action of the Weyl group of the $\Z_2$-graded symmetric Lie algebra
$\Zz_{ \k_0} (h_\k)\oplus (\Zz_{\g_1 \cap \pg} (h_\k))$.
\end{corollary}

The following theorem describes  the  set of orbits   of general mixed elements in $\g_1$. Recall that for
an element $e \in  \g_1$  we denote by $e_s + e_n$ its Jordan decomposition.

\begin{theorem}\label{orbit} Two elements $e_s +  e_n, e_s ' + e_n ' \in \g_1$ are in the same $Ad_{G_0}$-orbit,
if and only if $e_s$ belongs to the  orbit $Ad_{G_0}(e_s')$ and $e_n$ belongs to the orbit
 $Ad_{\Zz_{G _ 0}(e_s)}(e_n')$. 
\end{theorem}

Theorem \ref{orbit} is  straightforward, since the Jordan decomposition is  unique, see Theorem \ref{Theorem.2.1}.  We note that $Ad_{\Zz_{G_0} ( e_s)}$ may disconnected, but it  is a subgroup in the connected group $Ad_{\Zz_{G} ( e_s)}$ (by the Kostant theorem in \cite{Kostant1963}), so it seems possible to determine
this  subgroup. 

\begin{remark}\label{sumix} We summarize  our results in the following description of the set  of the   adjoint orbits in $\g_1$. Any element   in $\g_1$ is $Ad_{G_0}$-conjugate
to   an element of the form $h_\k + h_{\pg} + e_n$ such that \\
i) $h_\k$ is an elliptic semisimple element in $\h_{1\k}$,\\
ii)  $h_{\pg}$ is a real semisimple element,  commuting with $\h_\k$,\\
iii) $e_n$ is a nilpotent element,   commuting with $h_\k + h_{\pg}$.\\
Furthermore,  two elements $h_\k + h_{\pg} + e_n$ and $h_\k' + h_{\pg}' + e_n'$  are  conjugate,  only if
$h_\k$ is conjugate to $h_\k'$ under the action of the   associated Weyl group, see Corollary \ref{easy}.  Thus we can assume that $h_\k = h_\k '$.
Two elements  $h_\k + h_{\pg} +e_n$  and $h_\k + h_{\pg}' +e_n'$ are conjugate, only if $h_{\pg}$ and $h_{\pg}'$ are  conjugate under the action of the
associated Weyl group, see Corollary \ref{mix}.  Thus we can assume that  $h_{\pg} = h_{\pg}'$. Finally,  two elements $h_\k + h_{\pg} +e_n$  and $h_\k + h_{\pg} +e_n'$ are conjugate, if and only $e_n$ and $e_n'$ are  in the same orbit of nilpotent elements of the associated  $\Z_m$-graded reductive Lie algebra, see Theorem \ref{orbit}.  The classification of  these nilpotent orbits  can be obtained using the method in section 4.
\end{remark}

We finish this section by showing  the relation  between  the set of  orbits on real (resp. complex) $\Z_m$-graded Lie algebras and the $GL (8, \R)$-orbit spaces (resp.  the $GL (8, \C)$-orbit space)
of k-vectors  and k-forms on $\R ^ 8$ (resp. on $\C ^ 8$). To find a classification of  k-forms on $\R ^8$ is an important problem in classical invariant theory. 
Many interesting applications in geometry, \cite{DHM1988}, \cite{Hitchin2000}, \cite{LPV2008}, are related to this classification problem. This problem motivates  the author to write this note.

 Kac observed  that   the orbit space of homogeneous elements of degree 1
in  the $\Z_3$-graded complex algebra $\e_8$ (see example \ref{Example.3.3}.v) can be identified with the  $SL(9, \C)$-orbit space of 3-vectors on $\C ^ 9$, and  the orbit space of homogeneous elements of degree 1 in
the $\Z_2$-graded complex algebra $\e_7$ (see example \ref{Example.3.3}.iv) can be identified with  the orbit
space  of 4-vectors in  $\C ^ 8$ \cite{Kac1969}. In \cite{EV1978} Elashvili and Vinberg   classified all homogeneous elements
of degree 1 in  the $\Z_3$-graded Lie algebra $\e_8$. They also observed that, all 3-vectors in  
$\C ^ k$, $ k\le 8$, can be considered as  nilpotent elements  of degree 1  in this  $\Z_3$-graded Lie algebra
$\e_8$, furthermore a classification of $GL(k, \C)$-orbits on  $\Lambda ^ 3 (\C ^ k)$ is equivalent to a
classification of  these  homogeneous nilpotent elements.  In \cite{Djokovic1983}, based on this remark,
Djokovic classified all 3-vectors in $\C ^ 8$ and $\R ^ 8$. His classification is  reduced to a classification of homogeneous
nilpotent elements of degree 1 in a $\Z$-graded Lie algebra $\e_8$  (resp.$\e_{8 (8)}$).
His method is  close to our one (more precisely, our method is a generalization of his method), but
he used a  method of the Galois cohomology theory,  first used by Revoy in \cite{Revoy1979}, to compute the number of the open orbits
in $\Z$-graded $e_{8(8)}$.  Djokovic used the  Vinberg method of support to find
 a representative for each open orbit in $\Z$-graded $\e_{8(8)}$.

A classification of  4-vectors in $\C ^ 8$ has been given by Antonyan in \cite{Antonyan1981}. Using his classification
 and our  method  in this note it is   possible to classify all 4-vectors in $\R ^ 8$, which is reduced
 to the classification of  homogeneous elements of degree 1 in the $\Z_2$-graded Lie algebra $\e_{7(7)}$, (see example \ref{Example.3.3}.iv).

A classification of  $SL (9,\C)$-orbits of 3-forms on $\C ^ 9$ (resp.  $SL(9, \R)$-orbits
on $\Lambda ^ 3 (\R ^ 9) ^ *$) is equivalent to a classification
of homogeneous elements of  degree (-1) in the $\Z_3$-graded Lie algebra  $\e_8$ (resp. $\e_{8(8)}$) \cite{EV1978}.
By Corollary \ref{rev} this classification can be obtained from  a classification 3-vectors on $\C ^ 9$ (resp.
on $\R ^ 9$). In particular,  a classification of 3-forms on $\R ^ 8$  can be   obtained from the classification
of 3-vectors in $\R ^ 8$ in \cite{Djokovic1983}.

We note that a classification of $GL ( 8, \R)$-orbits on the space $\Lambda ^ k (\R ^ 8)$ can be obtained
easily from a classification of $SL (8,\R)$-orbits on the same space.

Given a volume element $vol ^ * \in \Lambda ^ 8 (\R ^ 8) ^ *$,  there is a unique element
 $vol_* \in\Lambda ^ 8 (\R  ^ 8)$   such that $\la vol ^ *, vol _* \ra = 1$.
Further there is a natural  Poincare isomorphism $P_* : \Lambda ^ k (\R  ^ 8)^ * \to \Lambda ^{8 -k} (\R ^ 8),
\:  \la P_* ( x), y\ra  = \la x\wedge y , vol_*\ra $, which commutes with the $SL (8, \R)$-action. 

Thus we can get a classification of all k-vectors and k-forms  on $\R ^ 8$ (resp. on $\C ^ 8$) using the theory of  real
(resp. complex) $\Z_m$-graded  semisimple Lie algebras.

{\bf Aknowledgement}.  The author is supported in part by grant IAA100190701 of Academy of Sciences of Czech Republic. She thanks Sasha Elashvili
for  stimulating helpful discussions   which   get her interested in this problem as well as for his help
in literature. She is grateful to  Saugata Basu, Gehard Pfister, Jiri Vanzura, Ernest Vinberg for helpful discussions,
and the anonymous referee for helpful remarks.   A part of this note has been written
during the author stay at the ASSMS of the Government College in Lahore-Pakistan,  
MSRI-Berkeley, and GIT-Atlanta. She  thanks these institutions for their hospitality and financial support.

\bigskip

{\small  Address: Mathematical Institute of ASCR,
Zitna 25, CZ-11567 Praha 1, 
email: hvle@math.cas.cz}

\end{document}